\documentclass[12pt]{amsart}
\usepackage{amsfonts}
\usepackage[cp1250]{inputenc}
\usepackage{amssymb}
\usepackage{amsmath}
\usepackage{amsthm}
\usepackage{latexsym}
\usepackage{amsbsy}
\usepackage{graphicx}
\usepackage{color}
\usepackage{geometry}
\usepackage[cp1250]{inputenc}
\usepackage{graphicx}
\usepackage{enumerate}

\newtheorem{theorem}{Theorem}
\newtheorem*{ttheorem}{Theorem C}

\newtheorem*{theoremI}{Hitt's Theorem}
\newtheorem*{theoremY}{Hayashi's Theorem}

\newtheorem*{theoremB}{Theorem B1}
\newtheorem*{theoremBB}{Theorem B2}

\newtheorem{proposition}{Proposition}

\newtheorem{lemma}{Lemma}
\newtheorem{corollary}{Corollary}
\theoremstyle{remark}

\theoremstyle{definition}

\def\D{\mathbb{D}}
\def\T{\mathbb{T}}

\def\H{\mathcal{H}}
\def\M{\mathcal{M}}%

\begin{document}
\title[Decompositions of de Branges-Rovnyak spaces]{Decompositions of de Branges-Rovnyak spaces}
\author{Bartosz \L anucha, Ma\l gorzata Michalska and Maria T. Nowak}
\address{Department of Mathematics\\ Maria Curie-Sk\l odowska University\\ Pl. Marii Curie-Sk\l odowskiej 1, 20-031 Lublin\\ Poland}%
\email{bartosz.lanucha@mail.umcs.pl}%
%\author{Ma\l gorzata Michalska}
%\address{Department of Mathematics\\ Maria Curie-Sk\l odowska University\\ Pl. Marii Curie-Sk\l odowskiej 1, 20-031 Lublin\\ Poland}%
\email{malgorzata.michalska@mail.umcs.pl}%
%\author{Maria T. Nowak}
%\address{Department of Mathematics\\ Maria Curie-Sk\l odowska University\\ Pl. Marii Curie-Sk\l odowskiej 1, 20-031 Lublin\\ Poland}%
\email{maria.nowak@mail.umcs.pl}%

%\submitted{}

\begin{abstract}
    We  obtain  orthogonal  decompositions for de Branges-Rovnyak spaces
     $\H\left( \tfrac {I^n(1+I)}{2}\right)$ and $\H\left( \tfrac {I(1+I^2)}{2}\right)$, where $I$ is an inner function. We also discuss  the problem of cyclicity for these spaces.
\end{abstract}

\maketitle

 \section{Introduction}

 Let $H^2$ be the standard Hardy space on the unit disk $\mathbb
 D=\{z\in\mathbb{C}:\ |z|<1\}$ and let $\mathbb T=\partial \mathbb D$. For $\varphi\in
 L^{\infty}(\mathbb T)$ the Toeplitz operator $T_{\varphi}$ on $H^2$ is given by
 $T_{\varphi}f=P(\varphi f)$, where $P$ is the orthogonal
 projection from $L^2(\mathbb T)$ onto $H^2$.

%It is known that $H^2$ can be identified via boundary functions with a closed subspace of $L^2=L^2(\mathbb{T})$.
%This subspace consists of functions whose Fourier coefficients with negative indices vanish.
%

Given a function $b$ in the unit ball of $H^{\infty}$, the
{de Branges-Rovnyak space} $\mathcal{H}(b)$ is the image of
$H^2$ under the operator $(1-T_bT_{\overline{b}})^{1/2}$ with the range norm $\|\cdot\|_b$ defined by $\|( 1- T_bT_{\overline{b}})^{1/2}f\|_b= \|f\|$ for $f\in(\text{ker}
 (1-T_bT_{\overline{b}})^\bot$, where $\|\cdot\|$ denotes the norm in $H^2$. For the basic theory of
  $\mathcal H(b)$ spaces we refer to the books \cite{Sarason1} and   \cite{FM}.

  It is known that $\mathcal{H}(b)$ is a Hilbert space with a reproducing kernel
$$k_{w}^b(z)=\frac{1-\overline{b(w)}b(z)}{1-\overline{w}z},\quad z,w\in\mathbb{D}.$$
If $b$ is an inner function, then   $\mathcal H(b)$ is the
well-known model space $K_b=H^2\ominus bH^2$.
If $b$ is a
nonextreme point of the unit ball of $H^{\infty}$ (equivalently, $\log(1-|b|)\in L^1(\T)$), then one can define an outer function $a\in H^{\infty}$
such that $|b|^2+|a|^2=1$ a.e. on $\mathbb{T}$ and we say that $(b,a)$ is a pair. Moreover, if we suppose that $a(0)>0$, then $a$ is uniquely determined. Then we  also consider the Hilbert spaces $\mathcal M(a)= aH^2$ and $ \mathcal M(\overline a)= T_{\overline a}H^2$ with the corresponding range norms.
For these spaces the following contractive inclusions hold $$ \mathcal M(a)\subset  \mathcal M(\overline a)\subset \mathcal H(b).$$
Moreover, if $(b,a)$ is a corona pair (that is,  $|a|+|b|$ is bounded away from $0$ in $\mathbb{D}$), then
$\mathcal{H}(b)=\mathcal{M}(\overline{a})$, as sets.

Let $S$ denote the unilateral shift on $H^2$, that is,  the Toeplitz operator $T_z$ on $H^2$.
It is known that for a nonextreme $b$ the space $\mathcal H (b)$ is invariant   under  $S$ and  $S^*$.
Let $Y$ be the restriction of $S$ to $\mathcal{H}(b)$.
Recently A. Bergman (\cite{AB})  obtained interesting   characterizations of cyclic vectors in
 spaces $\mathcal{H}(b)$  with respect to the operator $Y$. The description of these vectors  depends on the orthogonal complement
 of $\mathcal M(a)$ in $\mathcal H(b)$.

In \cite{FMS},\cite{FG}  and \cite {AB}
 the spaces $\mathcal H(b)$ with $b=\frac{1+I}2$,
 where $I$ is an inner function, have been studied
 and
 the following orthogonal decomposition has been derived
\begin{equation} \label{gwiazdka}\mathcal H\left(\tfrac{1+I}2\right)= \mathcal M\left(\tfrac{1-I}2\right)\oplus_b K_I.\end{equation}

In Section 3 of this paper we obtain orthogonal  decompositions  for spaces $\H\left( \tfrac {I^n(1+I)}{2}\right)$ and $\H\left( \tfrac {I(1+I^2)}{2}\right)$,
where $I$ is an inner function and $n$ is a positive integer. Our decompositions imply descriptions of orthogonal complements of the corresponding $\mathcal M(a)$ spaces. %We note that these decompositions give descriptions of orthogonal complements of the corresponding $\mathcal M(a)$ spaces.
%We also discuss the problem of cyclicity for these spaces.
We then show that $\H\left( \tfrac {I^n(1+I)}{2}\right)=\H\left( \tfrac {1+I}{2}\right)$ and $\H\left( \tfrac {I(1+I^2)}{2}\right)=\H\left( \tfrac {1+I^2}{2}\right)$ as sets. These equalities and
Bergman's results mentioned above allow us to describe cyclic vectors in the spaces under consideration.

Finally, we give another proof of A. Bergman's result on cyclic vectors in $\mathcal H(b)$ in the case when $\mathcal M(a)$ has a
 finite codimension, based on a result in \cite{NSS}.

\section{Preliminaries}

Since for a function $b$ in the closed unit ball of $H^\infty$ the function $\frac{1+b}{1-b}$ has a positive real part, there exists a positive Borel measure $\mu$ on $\mathbb{T}$ such that
\begin{equation}
\frac{1+b(z)}{1-b(z)}=\int_{\mathbb T}\frac {\zeta+ z}{\zeta-
    z}\,d\mu (\zeta)+i\text{Im}\frac{1+b(0)}{1-b(0)},\quad
z\in\mathbb D. \label{Herglotz}
\end{equation}
The measure $\mu$ is called the Aleksandrov-Clark measure associated with $b$. We note that in the case when $b$ is inner, the measure $\mu$ is singular  with respect to the normalized Lebesgue measure $m$. % On the other hand, if $(1-b)^{-1}\in L^2$, then the associated Aleksandrov-Clark measure $\mu$ is absolutely continuous (\cite{}).
For nonextreme $b$ the function $\bigl|F\bigr|^2$, where $F=\frac{a}{1-b}$, is the Radon-Nikodym derivative of the absolutely continuous component of
$\mu$ with respect to $m$. Hence the Lebesgue decomposition of $\mu$ is
$$d\mu=|F|^2dm+d\sigma,$$
where $\sigma$ is singular with respect to $m$.

If $\mu$ is the Aleksandrov-Clark measure associated with $b$, the operator $V_b$ defined by
\begin{equation}
\label{fau}
(V_bq)(z)= (1-b(z))\mathcal{K}_{\mu}q(z)= (1-b(z))  \int_{\mathbb T}\frac{q(\zeta)}{1-\bar{\zeta}z}d\mu(\zeta)\end{equation}
is an isometry of $H^2(\mu)$ (the closure of analytic polynomials in $L^2(\mu)$) onto $\mathcal H(b)$ (\cite[(III-7)]{Sarason1}).

It is also known that
$$V_{b}H^2(|F|^2dm)=(1-b)T_{\overline{F}}H^2$$
and $(1-b)T_{\overline{F}}$ is an isometry of $H^2$ into $\mathcal H(b)$.
 Moreover, $(1-b)T_{\overline{F}}$ maps $H^2$ onto $\mathcal H(b)$ if and only if $\mu$ is absolutely continuous (see \cite[(IV-13)]{Sarason1}
 for details). Furthermore, we have
\begin{align}
\begin{split}
\label{wz_H(b)}
V_{b}H^2(|F|^2dm)&= T_{1-b}T_{\overline{F}}\bigl( \overline{T_{\frac{F}{\overline{F}}}(H^2)} \oplus
\big(T_{\frac{F}{\overline{F}}}(H^2)\bigr)^\bot \bigr)\\ &=
\overline{\M(a)}^b\oplus_b T_{1-b}T_{\overline{F}}(\ker T_{\frac{\overline{F}}{F}}).
\end{split}
\end{align}
This means that in the case when the kernel of $T_{\frac{\overline{F}}{F}}$  is not trivial,  its  image under  $T_{1-b}$ will be contained in the orthogonal complement of $\mathcal M(a)$. This happens in the case when $b=\frac{I(1+I^2)}{2}$.

In the proof of Theorem 2 we apply Hayashi's theorem  that gives the characterization of the kernels of Toeplitz operators on $H^2$. Kernels of Toeplitz operators are the so-called nearly $S^\ast$-invariant spaces.
A closed subspace $M$ of $H^2$ is said to be \emph{nearly
$S^\ast$-invariant} if for every $f\in M$ vanishing at 0, we also
have $S^\ast f\in M$.
Nearly $S^\ast$-invariant spaces are characterized by Hitt's theorem
\cite{Hitt}.
\begin{theoremI}
The closed subspace M of $H^2$ is nearly $S^\ast$-invariant  if and
only if there exist
 a function f of unit norm  and a model space
$K_I=H^2\ominus IH^2$ such that $M= T_fK_I$, where $I$ is an inner
function vanishing  at the origin, and $T_f$ acts isometrically on
$K_I$.
\end{theoremI}
Assume that $f$ is an outer function of unit $H^2$ norm. Then  there exist a function $b$ in the unit ball of $H^{\infty}$ and an outer function $a$ such that $f=\frac{a}{1-b}$ and $|a|^2+ |b|^2=1$ a.e. on $\mathbb T$. It has been proved by D. Sarason (\cite{Sarason4})  that  $T_f$ acts isometrically on $ K_I$ if and only $I$ divides $b$. Then  $f=\frac{a}{1-b_0I}$ and  $f_0=\frac{a}{1-b_0}\in H^2$ (see  \cite{Sarason4}).
Recall that a function $f\in H^1$ is called rigid if and only if no other functions in $H^1$, except for a positive scalar multiples of $f$ have the same argument as $f$ a.e. on $\mathbb T$ .

Under the above notation Hayashi's theorem reads as follows

\begin{theoremY}
The nearly $S^\ast$-invariant space   $M=T_fK_I$ is the kernel of a
Toeplitz operator if and only if $f_0^2$ is a rigid function. Furthermore, then $M= \ker T_{\bar I\bar f/f}$.
\end{theoremY}

Next observe that for every $\lambda\in\mathbb{T}$ we have $\mathcal H(\overline{\lambda}b)=\mathcal H(b)$. For $\lambda\in\mathbb{T}$ let $\mu_{\lambda}$ denote the Aleksandrov-Clark measure associated with $\overline{\lambda}b$. So, if $F_{\lambda}=\frac{a}{1-\overline{\lambda}b}$, then $|F_{\lambda}|^2$ is the Radon-Nikodym derivative of the absolutely continuous component of
$\mu_{\lambda}$ with respect to $m$. Additionally note that the measure $\mu_{\lambda}$ is absolutely continuous for almost every $\lambda\in\mathbb{T}$ (\cite[(IV-10)]{Sarason1} and \cite[vol. II, p. 333]{FM}).

%
%It then follows that $\overline{\mathcal M(a)}^b=\H(b)$ if and only if $\ker T_{\frac{\bar
%       f_\lambda}{f_\lambda}}$ is trivial. It follows from the Hayashi Theorem that it happens if and only if $f^2_\lambda$ is rigid.

In the case when $\mathcal M(a)$ is not dense in $\H(b)$ we denote by $\H_0(b)$ the
orthogonal complement of $\M(a)$ in $\H(b)$. The dimension of $\H_0(b)$ is finite if and only if
there is a positive integer $k$ such that $\text{dim} \ker T_{\frac{\overline{F}_\lambda}{F_\lambda}}=k$. The Hayashi Theorem implies, that the last condition holds true if $F_\lambda=p\cdot f$, where $p$ is a polynomial of degree $k$ having all his roots
on the unit circle and $f^2$ is rigid (\cite[(X-17)]{Sarason1}). In this case the space $\H_0(b)$ has  been also  described explicitly in \cite{NSS}.

%
%the measure $\mu$ is absolutely continuous then $$\|\mu\|=\|F\|_2^2=\Re\left(\frac{1+b(0)}{1-b(0)}\right)$$%=\frac{1-|b(0)|^2}{|1-b(0)|^2}.$$

%We  also consider the Hilbert spaces $\mathcal M(a)= aH^2$ and $ \mathcal M(\overline a)= T_{\overline a}H^2$ with the corresponding range norms.
%For these spaces the following contractive inclusions hold $$ \mathcal M(a)\subset  \mathcal M(\overline a)\subset \mathcal H(b).$$
%Moreover, if $(b,a)$ is a corona pair (that is,  $|a|+|b|$ is bounded away from $0$ in $\mathbb{D}$), then
%$\mathcal{H}(b)=\mathcal{M}(\overline{a})$.

In the proof of Theorem \ref{uan} we use the following technical lemma, which may be of interest in itself.

\begin{lemma}
    \label{lem_zeros_of_pol}
    The polynomial \begin{equation}\label{peen}p_n(z)=a_nz^n+a_0\sum\limits_{j=0}^{n-1}z^j\quad\text{with } a_0>a_n>0,\end{equation} has no zeros in the closed unit disk $\overline{\mathbb{D}}$.
\end{lemma}
\begin{proof}
    We will apply the following Cohn's rule \cite[Thm. 11.5.3]{RS}.

    If  $r_n(z)=a_nz^n+\ldots+a_0$ is a complex polynomial such that $|a_n|<|a_0|$, then $r_n$ has the same number of zeros in $\overline{\mathbb{D}}$ as the polynomial $r_{n-1}=\overline{a}_0\cdot r_n-a_n\cdot r_n^*$, where
    $$r_n^*(z)=\overline{a}_0z^n+\ldots+\overline{a}_n.$$

    Using this rule to $p_n$ given by \eqref{peen} we see that the polynomial %$p_n(z)=a_nz^n+a_0\sum\limits_{j=0}^{n-1}z^j$, $a_0>a_n>0$, we see that the polynomial
    \begin{align*}
    p_{n-1}(z)&=a_0\cdot p_n(z)-a_n\cdot p_n^*(z)
    = a_0\cdot a_nz^n+\left(a_0\right)^2\sum_{j=0}^{n-1}z^j-a_n\cdot a_0\sum_{j=1}^{n}z^j-\left(a_n\right)^2\\
    &=(a_0-a_n) a_0\sum_{j=1}^{n-1}z^j+\left(a_0\right)^2 -\left(a_n\right)^2
    =(a_0-a_n) a_0\sum_{j=1}^{n-1}z^j+(a_0-a_n)(a_0+a_n)\\  &=a_{n-1}'\sum_{j=1}^{n-1}z^j+a_0',
    \end{align*}
    where
    $$a_{n-1}'=(a_0-a_n)a_0\quad\text{and}\quad a_0'=(a_0-a_n)(a_0+a_n),$$
    has the same number of zeros in $\overline{\mathbb{D}}$ as $p_n$. Also note that $0<a_{n-1}'<a_0'$.

    Now starting from $p_{n-1}$ we repeat this procedure to obtain polynomials $p_{n-2},\ldots, p_1$, each of the form
    \begin{equation*}p_k(z)=b_{k}\sum\limits_{j=1}^{k}z^j+b_0\quad\text{with } b_0>b_k>0,\quad k=1,2,\ldots,n-2,
    \end{equation*}
    having the same numbers of zeros in $\overline{\mathbb{D}}$ as $p_n$.
    Since $p_1(z)=b'_1z+b'_0$  has no zeros in $\overline{\mathbb{D}}$ ($b'_0>b'_1>0$), the same is true for $p_n$.
\end{proof}

 \section{The spaces $\H\left( \tfrac {I^n(1+I)}{2}\right)$ and $\H\left( \tfrac {I(1+I^2)}{2}\right)$}

%Assume first that $b=\tfrac {I^n(1+I)}{2}$, where $I$ is an inner function and $n\geq 0$ is an integer. Clearly, if $a=\tfrac {1-I}{2}$, then $(b,a)$ is a pair.
Our first result is the following.

%grivaux - pair up to a constant
\begin{theorem}\label{uan}
Let $b=\frac{I^n(1+I)}{2}$, where $I$ is an inner function and $n\geq 0$ is an integer. Then the following orthogonal decomposition of $\mathcal{H}(b)$ holds
    \begin{equation}\label{ququ}\mathcal{H}(b)= (1-I)H^2\oplus_b(2+2I+\ldots+2I^{n-1}+ I^n) K_I.\end{equation}
\end{theorem}
%Let $I$ be an inner function such that $I(0)=0$ and for an integer $n\geq 0$ let  $b=\frac{I^n(1+I)}{2}$. Then the following orthogonal decomposition of $\mathcal{H}(b)$ holds

Note that for $b$ given in Theorem \ref{uan} and $a=\tfrac {1-I}{2}$, $(b,a)$ is a pair.

\begin{proof}[Proof of Theorem \ref{uan}]
    By the Herglotz formula there exists a unique positive Borel measure $\mu$ on $\mathbb{T}$ such that
\begin{equation*}%\label{her}
\begin{split}
\frac{1+b(z)}{1-b(z)}
&=\int_{\mathbb{T}}\frac{\zeta+z}{\zeta-z}d\mu(\zeta)+i\text{Im}\frac{1+b(0)}{1-b(0)}.
%&=\frac{1}{2\pi}\int_0^{2\pi}\frac{1+ze^{-it}}{1-ze^{-it}}|F(e^{it})|^2\,dt
%+\frac{1}{2\pi} \int_0^{2\pi}\frac{1+ze^{-it}}{1-ze^{-it}}d\sigma(e^{it}).
\end{split}
\end{equation*}
%Note that since $b(0)=0$, we have $\mu(\mathbb{T})=1$.
We also know that the absolutely continuous component of the measure $\mu$ is equal to $|F|^2dm$, where
\begin{equation}
\label{ef}F=\frac{a}{1-b}=\frac{1-I}{2-I^n-I^{n+1}}=\frac{1}{I^n+2I^{n-1}+\ldots+2I+2}.\end{equation} We will show that
\begin{equation}
\label{compo}
d\mu=|F|^2dm+\tfrac{2}{2n+1}\,d\sigma_I,
\end{equation}
where $\sigma_I$ is the singular measure such that
$$ \frac{1-|I|^2}{|1-I|^2}=
 \int_{\mathbb{T}}\frac{1-|z|^2}{|\zeta-z|^2}\, d\sigma_I(\zeta).$$
The last equality means also that
 $$ \frac{1+I}{1-I}=
 \int_{\mathbb{T}}\frac{\zeta+z}{\zeta-z}\, d\sigma_I(\zeta)+i\text{Im}\frac{1+I(0)}{1-I(0)}$$
 and so \eqref{compo} is equivalent to
\begin{equation*}\label{two}
\frac{1+b}{1-b}=\int_{\mathbb{T}}\frac{\zeta+z}{\zeta-z}|F(\zeta)|^2dm(\zeta)+\frac{2}{2n+1}\cdot\frac{1+I}{1-I}+i\text{Im}\frac{1+b(0)}{1-b(0)}-\frac{2i}{2n+1}\text{Im}\frac{1+I(0)}{1-I(0)}.
\end{equation*}
A calculation gives
$$\widetilde{F}=\frac{1+b}{1-b}-\frac{2}{2n+1}\cdot\frac{1+I}{1-I}
=\frac{-(2n-1)I^n+\sum_{j=0}^{n-1}(4n-2-8j)I^j}{(2n+1)(I^n+2I^{n-1}+\ldots+2I+2)}$$
and Lemma \ref{lem_zeros_of_pol} implies that $\widetilde{F}\in H^\infty$. % Moreover, we have $$\text{Re} \widetilde{F}(0)=\text{Re}\frac{1+b(0)}{1-b(0)}-\frac{2}{2n+1}\text{Re}\frac{1+I(0)}{1-I(0)}$$%=\tfrac{2n-1}{2n+1}=$$
%and,
 Since $\text{Re} \frac{1+I}{1-I}=0$ a.e. on $\mathbb{T}$,
$$\text{Re} \widetilde{F}=\text{Re} \frac{1+b}{1-b}=|F|^2\quad \text{a.e. on }\mathbb{T}.$$
This proves \eqref{compo}. Consequently, by \eqref{fau},
\begin{align*}
\mathcal{H}(b)&=V_{b}H^2(d\mu)=V_{b}H^2(|F|^2dm)\oplus_{b}V_{b}L^2(d\sigma_{I}).%\\
%&=(1-b)T_{\overline{F}}H^2\oplus_b V_bL^2({d\sigma_I})
%=\tfrac12(1-I)H^2\oplus_b V_bL^2({d\sigma_I}).
\end{align*}

Now we note that if $F$ is given by \eqref{ef}, then $F^2$ is a rigid function.  This follows from Theorem 6.22 in \cite{FM}, which states that if a function $f\in H^1$ is such that $1/f\in H^1$, then $f$ is rigid.

Hence $\text{ker}T_{\frac{\overline{F}}{{F}}}=\{0\}$ (see \cite{Sarason4}) and it follows that
$$V_{b}H^2(|F|^2dm)=(1-b)T_{\overline{F}}H^2=(1-b)T_{\overline{F}}\left(\overline{{T_{\frac{{F}}{\overline{F}}}H^2}}\oplus\text{ker}T_{\frac{\overline{F}}{{F}}}\right)=(1-I)H^2.$$
Furthermore, observe that
\begin{equation}
\label{12}
\begin{split}V_bL^2({d\sigma_I})&=(1-b)\mathcal{K}_{\sigma_I}(L^2({d\sigma_I}))\\
&=\tfrac12(I^n+2I^{n-1}+\ldots+2I+2)(1-I)\mathcal{K}_{\sigma_I}(L^2({d\sigma_I})),\end{split}
\end{equation}
where
$$\mathcal{K}_{\sigma_I}h(z)=\int_{\mathbb{T}}\frac{h(\zeta)}{1-\overline{\zeta}z}d\sigma_I(\zeta),\quad h\in L^2({d\sigma_I}).$$
Since
$$(1-I)\mathcal{K}_{\sigma_I}(L^2({d\sigma_I}))=V_IL^2({d\sigma_I})= K_I,$$
the proof of \eqref{ququ} is complete.

%Hence
%$$\widetilde{F}(z) =\int_{\mathbb{T}}\frac{\zeta+z}{\zeta-z}\text{Re} \widetilde{F}(\zeta)\,dm(\zeta)
%=\int_{\mathbb{T}}\frac{\zeta+z}{\zeta-z} |F(\zeta)|^2dm(\zeta).$$
\end{proof}

%\begin{remark}
%We mention here that if $b$ is as in Theorem \ref{uan} with $I(z)=z$, then $b'(1)=\frac{2n+1}{2}$.
%\end{remark}

\begin{proposition}\label{tuan}
    For every nonconstant inner function $I$ and for any positive integer $n$, we have
    \begin{equation}\label{rowne}
    \mathcal{H}(\tfrac{I^n(1+I)}{2})=\mathcal{H}(\tfrac{1+I}{2})
    \end{equation}
    as sets.
\end{proposition}

\begin{proof}
     To show \eqref{rowne} we recall that for $b_1$, $b_2$ from the closed unit ball of $H^\infty$ we have
    \begin{equation}\label{dec}
    \mathcal{H}(b_1b_2)=\mathcal{H}(b_1)+b_1\mathcal{H}(b_2).
    \end{equation}
    Moreover, if $b_1$ is an inner function, then the decomposition \eqref{dec} becomes orthogonal  (see \cite[(I-10)]{Sarason1} and \cite[(II-6)]{Sarason1}). Hence in our case
    $$\mathcal{H}(\tfrac{I^n(1+I)}{2})=K_{I^n}\oplus_b I^n\mathcal{H}(\tfrac{1+I}{2}).$$
    Recently E. Fricain and S. Grivaux \cite[Lemma 2.5]{FG} proved the following inclusions
    $$ K_{I^n}\subset\mathcal{H}(\tfrac{1+I}{2})\quad\text{and}\quad  I^n\mathcal{H}(\tfrac{1+I}{2})\subset \mathcal{H}(\tfrac{1+I}{2}).$$
Therefore
    $$\mathcal{H}(\tfrac{I^n(1+I)}{2})\subset \mathcal{H}(\tfrac{1+I}{2}).$$
    For the reverse inclusion it is enough to note that by \eqref{dec},
    $$\mathcal{H}(\tfrac{I^n(1+I)}{2})=\mathcal{H}(\tfrac{1+I}{2})+(1+I)K_{I^n}.    $$
%   and \eqref{rowne} follows.

\end{proof}

%%%%%%%%%%%%%%%%%%%%%%%%%%%%%%%%%%%%%%%%%%%%%%%%%%%%%%%%%%

\begin{theorem}\label{too}
    Let $I$ be an inner function such that $I(0)=0$ and let  $b=\frac{I(1+I^2)}{2}$. Then the following orthogonal decomposition of $\mathcal{H}(b)$ holds
    \begin{equation}\label{quququ}\mathcal{H}(b)= {(1-I^2)H^2}\oplus_b(2-I)(1-I)K_I\oplus_b(2+I+I^2) K_I.\end{equation}%(I^2-3I+2)
\end{theorem}
\begin{proof}
As before, we have a unique positive Borel measure $\mu$ on $\mathbb{T}$ such that
\begin{equation*}%\label{her}
\begin{split}
\frac{1+b(z)}{1-b(z)}
&=\int_{\mathbb{T}}\frac{\zeta+z}{\zeta-z}d\mu(\zeta).
%&=\frac{1}{2\pi}\int_0^{2\pi}\frac{1+ze^{-it}}{1-ze^{-it}}|F(e^{it})|^2\,dt
%+\frac{1}{2\pi} \int_0^{2\pi}\frac{1+ze^{-it}}{1-ze^{-it}}d\sigma(e^{it}).
\end{split}
\end{equation*}
It is worth noting that $\mu$ is a probablility measure since $b(0)=0$. In this case the absolutely continuous component of the measure $\mu$ is equal to $|F|^2dm$, where
 \begin{equation*}
 \label{eff}F=\frac{a}{1-b}=\frac{1-I^2}{2-I-I^3}=\frac{1+I}{2+I+I^2}\end{equation*}
 and we show that
 \begin{equation}
 \label{compoz}
 d\mu=|F|^2dm+\tfrac{1}{2}\,d\sigma_I,
 \end{equation}
 where $\sigma_I$ is the singular measure defined as in the proof of Theorem \ref{uan}.
 Analogously, we define the function
 \begin{equation*}\label{eq_ex2_1}
 \widetilde{F}=\frac{1+b}{1-b}-\frac{1}{2}\cdot\frac{1+I}{1-I}=\frac{(2-I)(1+I)}{2(2+I+I^2)}\in H^{\infty}.
 \end{equation*}
Moreover, we have $\widetilde{F}(0) =\text{Re} \widetilde{F}(0)=\tfrac{1}{2}$ and $\text{Re} \widetilde{F}=|F|^2$ a.e. on $\mathbb{T}$, which proves \eqref{compoz}. Hence
 \begin{align*}
 \mathcal{H}(b)&=V_{b}H^2(|F|^2dm)\oplus_{b}V_{b}L^2(d\sigma_{I}).%\\
 %&=(1-b)T_{\overline{F}}H^2\oplus_b V_bL^2({d\sigma_I})
 %=\tfrac12(1-I)H^2\oplus_b V_bL^2({d\sigma_I}).
 \end{align*}
 Analogously to \eqref{12}, we get
 $$V_bL^2({d\sigma_I})=\tfrac12(2+I+I^2)(1-I)\mathcal{K}_{\sigma_I}(L^2({d\sigma_I}))=(2+I+I^2)K_I.$$
Furthermore,
 \begin{align*}V_{b}H^2(|F|^2dm)&=(1-b)T_{\overline{F}}H^2 =(1-b)T_{\overline{F}}\big(\overline{T_{\frac{F}{\overline{F}}}H^2}^2 \oplus \ker T_{\frac{\overline{F}}{F}}\big)\\&=
{ (1-I^2)H^2}\oplus_b (1-b)T_{\overline{F}}\big(\ker T_{\frac{\overline{F}}{F}}\big). \end{align*}
 To compute $\ker T_{\frac{\overline{F}}{F}}$ observe first that
 $$\frac{\overline{F}}{F}=\overline{I}\frac{\overline{f}_1}{f_1},$$
 where $f_1=\tfrac{1}{2+I+I^2}$. It is easy to check that
 $$\text{Re}\left(\frac{-3I^2-I+6}{2+I+I^2}\right)=8|f_1|^2\quad\text{a.e. on }\mathbb{T},$$
% Indeed,
% \begin{align*}
%\text{Re}\left(\frac{-3I^2-I+6}{2+I+I^2}\right)=
%\frac{\text{Re}\big((-3I^2-I+6)(2+\overline{I}+\overline{I}^2)\big)}{|2+I+I^2|^2}=\frac{8}{|2+I+I^2|^2}
% \end{align*}
 and so
$ \|f_1\|^2=\tfrac38$. Now let $f=\sqrt{\tfrac{8}{3}}f_1$. This outer function can be written in the form $f=\frac{a_0}{1-Ib_0}$ with
 $$a_0=\frac{2\sqrt{6}}{6+I} \qquad \text{and}\qquad b_0=-\frac{2+3I}{6+I}.$$
 Since
 $$f_0^2=\left(\frac{a_0}{1-b_0}\right)^2=\tfrac{3}{2}\frac{1}{(I+2)^2}$$ is rigid, by Hayashi's theorem (see \cite{Hayashi1,Hayashi2} and \cite{Sarason4}) we get
 $$\ker T_{\frac{\overline{F}}{F}}=\ker T_{\overline{I}\frac{\overline{f}}{f}}=f K_I.$$

 Now, for any function $h\in K_I$,
 \begin{align*}
    T_{\overline{F}}(fh)&=P(\overline{F}fh)=\sqrt{\tfrac{8}{3}}P\left(\frac{1+\overline{I}}{|2+I+I^2|^2}h\right)=\sqrt{\tfrac{1}{24}}P\left(({1+\overline{I}})h\text{Re}\left(\frac{-3I^2-I+6}{2+I+I^2}\right)\right)\\
    &=\tfrac{1}{4\sqrt{6}}P\left(({1+\overline{I}})h\left(\frac{-3I^2-I+6}{2+I+I^2}+\frac{-3\overline{I}^2-\overline{I}+6}{2+\overline{I}+\overline{I}^2}\right)\right).
 \end{align*}
 Hence, using the fact that $\overline{I}h\in \overline{zH^2}$ for $h\in K_I$, we get
 $$T_{\overline{F}}(fK_I)=(2-I)fK_I.$$
Consequently,
 \begin{align*}
 (1-b)T_{\overline{F}}\big(\ker T_{\frac{\overline{F}}{F}}\big)&=(1-b)T_{\overline{F}}\big(fK_I\big)=(1-b)(2-I)fK_I\\
 &=(2-I)(1-I)K_I=(I^2-3I+2)K_I.
 \end{align*}
% Summing up
% $$\mathcal{H}(b)\ominus \mathcal{M}(a)=(2+I+I^2) K_I\oplus_b(I^2-3I+2)K_I.$$
 %$$\mathcal{H}(b)\ominus \mathcal{M}(a)=V_bL^2({d\sigma})=\tfrac{1}{2n+1}p_n(I)\cdot V_I(L^2({d\sigma_I}))=p_n(I)\cdot K_I.$$

%$$\mathcal{H}(b)=V_{b}H^2(d\mu),$$

\end{proof}

\begin{proposition}

    For every nonconstant inner function $I$ such that $I(0)=0$ we have
        \begin{equation}\label{incl}
        \mathcal{H}(\tfrac{I(1+I^2)}{2})=\mathcal{H}(\tfrac{1+I^2}{2}).
        \end{equation}
    as sets.
\end{proposition}
\begin{proof}
    As in the proof of Proposition \ref{tuan},
    \begin{equation*}
    \mathcal{H}(\tfrac{1+I^2}{2})\subset \mathcal{H}(\tfrac{1+I^2}{2})+\tfrac{1+I^2}{2}K_I=\mathcal{H}(\tfrac{I(1+I^2)}{2}).
    \end{equation*}
For the reverse inclusion recall that by \eqref{gwiazdka} and \eqref{dec},
    $$\mathcal{H}(\tfrac{1+I^2}{2})=(1-I^2)H^2+K_{I^2}=(1-I^2)H^2+K_{I}+IK_{I}.$$
    Since
    $$2+I+I^2=-(1-I^2)+3+I,$$
     we have
    $$(2+I+I^2)K_I\subset (1-I^2)K_I+K_{I}+IK_{I}=\mathcal{H}(\tfrac{1+I^2}{2}).$$
    Similarly,
    $$(I^2-3I+2)K_I\subset\mathcal{H}(\tfrac{1+I^2}{2})$$
    and \eqref{quququ} implies
     $$\mathcal{H}(\tfrac{I(1+I^2)}{2})\subset\mathcal{H}(\tfrac{1+I^2}{2}).$$%=(1-I^2)H^2+(2+I+I^2) K_I+(I^2-3I+2)K_I
%   and \eqref{incl} follows.
\end{proof}

\section{Remarks on cyclicity}

In this section we study the problem of cyclicity for the forward shift $Y$ on $\mathcal{H}(b)$ associated to a nonextreme $b$. For $f\in \mathcal{H}(b)$ let $[f]$ denote the closed linear span of polynomial multiples of $f$ in $\mathcal{H}(b)$. A function $f$ is called cyclic for $\mathcal{H}(b)$ if $[f]=\mathcal{H}(b)$. % We say that $f\in\mathcal{H}(b)$ is cyclic (for $Y$) if the linear span of $\{Y^nf:\ n\geq 0\}=\{z^nf:\ n\geq 0\}$ is dense in $\mathcal{H}(b)$.
Since in the nonextreme case polynomials are dense in $\mathcal{H}(b)$, a function $f$ is cyclic if and only if there exists a sequence of analytic polynomials $p_n$, such that \begin{equation}\label{ence}\lim\limits_{n\to\infty}\|1-p_nf\|_b=0.\end{equation}
Since $\mathcal{H}(b)$ is contained contractively in $H^2$, if $f$ is cyclic for $\mathcal{H}(b)$, then $f$ is an outer function.

%Let $V_b$ be given by \eqref{fau}. It follows from results of Poltoratski (see \cite{AP} and \cite{CMR}) that for $q\in L^{1}(d\mu)$ the function $V_bq$ has
%non-tangential limits $\mu$-a.e and that with respect to the singular component of $\mu$ these limits equal $q$.

Recently, using the Poltoratski result (see \cite{AP} and \cite[p. 231]{CMR}), A. Bergman obtained the following characterization of cyclic vectors for $\mathcal{H}(\tfrac{1+I}{2})$ in terms of the Aleksandrov-Clark measure $\sigma_I$ associated with the inner function $I$ (defined as in the proof of Theorem~\ref{uan}).
\begin{theoremB}[\cite{AB}]\label{be}
    Let $I$ be an inner function and let $b=\frac{1+I}{2}$. An outer function $f\in\mathcal{H}(b)$ is cyclic if and only if $f\neq 0$,  $\sigma_I$-a.e. on $\mathbb{T}$.
\end{theoremB}

Here we get
\begin{corollary}
    Let $I$ be an inner function. Then
    \begin{enumerate}[(a)]
    \item an outer function $f\in\mathcal{H}(\tfrac{I^n(1+I)}{2})$ is cyclic if and only if $f\neq 0$  $\sigma_I$-a.e. on $\mathbb{T}$.
        \item an outer function $f\in\mathcal{H}(\tfrac{I(1+I^2)}{2})$, with $I(0)=0$, is cyclic if and only if $f\neq 0$  $\sigma_{I^2}$-a.e. on $\mathbb{T}$.
\end{enumerate}
\end{corollary}
\begin{proof}
To prove (a) it is enough to note that, by Proposition \ref{tuan},  $\mathcal{H}(\tfrac{I^n(1+I)}{2})=\mathcal{H}(\tfrac{1+I}{2})$, and use \mbox{Theorem B1}.
Indeed, both of these spaces are contractively contained in $H^2$ and the closed graph theorem implies that their norms are equivalent.
Hence \eqref{ence} holds true for $b=\tfrac{I^n(1+I)}{2}$ if and only if it does for $b=\tfrac{1+I}{2}$.
%
%cyclicity is defined in terms of convergence, and so it does not depend on the equivalent norm one considers. Hence (a) follows from Theorem B.

The proof of (b) is analogous.
\end{proof}

%If $\mathcal M(a)$ is not dense in $\mathcal H(b)$,  we let $\mathcal H_0(b)$ denote the orthogonal complement  of
%$\mathcal M(a)$ in $\mathcal H(b)$.
It is shown in \cite{Sarason1} that the
codimension of ${\M(a)}$ in $\H(b)$ is $N$  if  and only if
the operator $Y^\ast$ has eigenvalues $\overline{z}_1, \overline{z}_2,\dots, \overline{z}_s$ on the
unit circle with their algebraic multiplicities $n_1,\dots, n_s$ and
$N=n_1+n_2+\dots+n_s$. Then $\H_0(b)$ is the span of the root spaces of the operator $Y^\ast$, that is,
$
\H_0(b)=\bigvee_{j=1}^s \ker(Y^\ast-\bar
{z}_j)^{n_j}.
$
In \cite{NSS} this space has been described explicitly. To cite this result we need additional notation.

%
%Recall that if $\mathcal M(a)$ is not dense in $\mathcal H(b)$,  we let $\mathcal H_0(b)$ denote the orthogonal complement  of
%$\mathcal M(a)$ in $\mathcal H(b)$. Assume now that the
%codimension of $\overline{\M(a)}$ in $\H(b)$ is $N$. As mentioned in the Introduction, $\H_0(b)$ is the span of the root spaces of the operator $Y^\ast$, that is,
%$$
%\H_0(b)=\bigvee_{j=1}^s \ker(Y^\ast-\bar
%{z}_j)^{n_j},
%$$
%where  $\overline{z}_1, \overline{z}_2,\dots, \overline{z}_s\in\mathbb{T}$ are the eigenvalues of $Y^\ast$ with algebraic multiplicities $n_1,\dots, n_s$,
%$N=n_1+n_2+\dots+n_s$.

For  $k=0,1,\dots $, let
$\upsilon^k_{b,w}=\frac{\partial^k k_{w}^b}{\partial \bar{w} ^k}$ be the kernel
function in $\H(b)$  for the functional of evaluation of the $k$-th
derivative at $w\in\mathbb{D}$, which means that for $f\in\H(b)$,
\[
f^{(k)}(w)= \langle f,\upsilon^k_{b,w}\rangle_b.
\]
Under  the above assumption the functions $\upsilon^{k_j}_{b,z_j}$ ($j=1, 2,\dots, s,\ k_j= 0,1,\dots n_j-1$)
can be defined as pointwise limits of $ \upsilon^{k_j}_{b,w}$ as $w$ tends nontangentially to $z_j$, that is,
\[
\upsilon^{k_j}_{b,z_j}(z)=\lim_{\substack{w\to z_j \\
        \sphericalangle}}\upsilon^{k_j}_{b,w}(z),\quad  z\in\D.
\]
Moreover, for any $f\in\mathcal H(b)$,
\[
f^{(k_j)}(z_j)=\lim_{\substack{w\to z_j \\
        \sphericalangle}}f^{(k_j)}(w)= \lim_{\substack{w\to z_j \\
        \sphericalangle}}\langle f,  \upsilon^{k_j}_{b,w}\rangle_b=\langle f,  \upsilon^k_{b,z_j}\rangle_b.
\]

In \cite{NSS} the following explicit description of $\H_0(b)$ has been obtained for the above case.

\begin{ttheorem}[\cite{NSS}]\label{tw2} Assume that $\dim\H_0(b)=N$ and the operator $Y^\ast$ has eigenvalues $\overline{z}_1, \overline{z}_2,\dots, \overline{z}_s$ on the
    unit circle with their algebraic multiplicities $n_1,\dots, n_s$ and
    $N=n_1+n_2+\dots+n_s$. Then, under the above notation,
    \begin{equation}
        \H_0(b)=\bigvee\limits\{\upsilon_{b,z_j}^{k_j}\colon\,j=1,\ldots,s,\  k_j=0,1,\ldots,n_j-1\}.
        \label{H0}\end{equation}
\end{ttheorem}

%We are interested in cyclic vectors for the operator $Y: \H(b)\to \H(b)$ ( recall that  $Y$ is the
%restriction of the shift operator $S$ to $\H(b)$).
%It is known that if $f\in \H(b)$ is cyclic for $Y$, then $f$ is  an outer function.
%Moreover, any outer function $f\in \H(b)$ is cyclic for $\H(b)$ if and only if $\mathcal M(a)$ is dense in $\H(b)$ (see [AB] and [FMS])
%
%Important results on such cyclic vectors  are given in [AB].
%Among other results the autor obtains the  characterizations  of cyclic vectors in the case
%when the codimension of $\overline{\M(a)}$ in $\H(b)$ is finite. His result can be stated as  follows

Recently, A. Bergman \cite{AB} proved also
\begin{theoremBB}[\cite{AB}]

    Under the assumptions of Theorem C a function $f\in \H(b)$ is cyclic if and only if $f$ is outer and $f(z_j)\ne 0$, $j=1,2,\dots,s$.
\end{theoremBB}
Using Theorem C we can give a slightly different proof of the sufficiency of the condition for cyclicity in this case.

To this end assume that  a nonextreme $b$ is such that $\H(b)=\mathcal M(a)\oplus_b \H_0(b)$ and $\H_0(b)$ is given by (\ref{H0}).
Let  $f\in \H(b)$ be  an outer function such that $f(z_j)\ne 0$, $ j=1,2 \dots ,s$.
It is enough to show that if $h\in \H(b)$ and for $ n=0,1, \dots$,
\[\langle h, z^nf\rangle_b=0,\] then  $h=0$.  Let $h=h_a+\tilde  h$, where $h_a\in\mathcal M(a)$ and $\tilde  h\in \mathcal H_0(b).$
Since $\mathcal M(a)\subset [f]$, we can assume that $ h\in \mathcal H_0(b)$.

First assume  additionally that $$ \H_0(b)=\bigvee\limits\{\upsilon_{b,z_1}^{k}\colon\,  k=0,1,\ldots,n_1-1\},$$ which means that  the   operator  $Y^\ast$ has only one eigenvalue $\overline{z}_1$
with the multiplicity $n_1$ ($n_1\geq 1$).
Then an $h\in\mathcal H_0(b)$ can be written  as
\begin{equation} h= c_0 \upsilon_{b,z_1}^{0 }+ c_1\upsilon_{b,z_1}^{1 }+ \dots + c_{n_1 -1}\upsilon_{b,z_1}^{n_1-1 }.\label{1}\end{equation}
Assume that an outer $f\in \mathcal H(b)$  is such that    $f(z_1)\ne 0$ and  $\langle h,z^nf\rangle= 0  $ for $n=0,1,2,\dots$.  Then, in particular,
\[
0=\langle c_0\upsilon_{b,z_1}^{0 }+ c_1\upsilon_{b,z_1}^{1 }+ \dots + c_{n_1 -1}\upsilon_{b,z_1}^{n_1-1 }, (z-z_1)^{n_1-1}f\rangle_b
%= c_{n_1-1}\lim_{\substack{z\to z_j \\
    %   \sphericalangle}}\langle \upsilon_{b,z_1}^{n_1-1 }, (z-z_1)^{n_1-1}f\rangle_b
\]
Now notice that for $k=0,1,\ldots,n_1-2$,
$$\langle \upsilon_{b,z_1}^{k }, (z-z_1)^{n_1-1}f\rangle_b=0$$
and
\begin{align*}%\lim_{\substack{z\to z_1 \\
%\sphericalangle}}
&\overline{  \langle \upsilon_{b,z_1}^{n_1-1 }, (z-z_1)^{n_1-1}f\rangle_b} =\lim_{\substack{z\to z_1 \\
            \sphericalangle}}\left( (z-z_1)^{n_1-1}f(z)\right)^{(n_1-1)}\\ &=\lim_{\substack{z\to z_1\\
            \sphericalangle}}\left(\sum_{k=0}^{n_1-1}\binom{n_1-1}{k}(n_1-1)(n_1-2)\cdots (n_1-k)(z-z_1)^{n_1-1-k}f^{(n_1-1-k)}(z)\right)\\ &=
    \lim_{\substack{z\to z_1 \\
            \sphericalangle}} f(z)(n_1-1)!=f(z_1)(n_1-1)!\neq 0,
\end{align*}
which implies  $ c_{n_1-1}=0$.

Next applying similar reasoning to $ c_0\upsilon_{b,z_1}^{0 }+ c_1\upsilon_{b,z_1}^{1 }+ \dots + c_{n_1 -2}\upsilon_{b,z_1}^{n_1-2 }$ with
$(z-z_1)^{n_1-1}f$ replaced by $(z-z_1)^{n_1-2}f$ we can show that $c_{n_1-2}=0 $ and  continuing this procedure we find that $h=0$.

Similar reasoning can be also applied to the general case. Indeed, if $\mathcal H_0(b)$ is given by (\ref{H0}),  $h\in\mathcal H_0(b)$
and $f\in \mathcal H(b)$ is an outer function  such that $f(z_j)\ne0$  for $j=1,2,\dots ,s$,
then we can note first that
if $ h= \sum\limits_{j=1}^{s}\sum\limits_{k_j=0}^{n_s-1}c_{j,k_j}\upsilon_{b,z_j}^{k_j }$ and

$$\langle h,(z-z_1)^{n_1}(z-z_2)^{n_2}\dots (z-z_s)^{n_s-1}f\rangle_b=0,$$ then $c_{s,n_s-1}=0$ and apply an analogous procedure to that above to show that $h=0$.

 \subsection*{Acknowledgments}

 The authors want to thank Professor Andrzej So\l tysiak for helpful discussions while preparing this article.

\end{document}